\documentclass{article}

\usepackage{arxiv}
\usepackage[utf8]{inputenc} 
\usepackage[T1]{fontenc}    
\usepackage{url}            
\usepackage{booktabs}       
\usepackage{amsfonts}       
\usepackage{nicefrac}       
\usepackage{microtype}      
\usepackage{amsfonts}
\usepackage{amsthm}
\usepackage{amsmath}
\usepackage{amscd}
\usepackage{graphicx}
\usepackage{amssymb}
\usepackage{color}
\usepackage{tikz}
\usepackage[numbers]{natbib}
\usepackage[ruled,vlined]{algorithm2e}
\usetikzlibrary{positioning}
\usetikzlibrary{arrows}
\usepackage{algorithmic}
\newtheorem{remark}{Remark}
\newtheorem{theorem}{Theorem}[section]
\newtheorem{lemma}{Lemma}[section]
\usepackage{multirow}
\usepackage{tikz,bm}
\usetikzlibrary{arrows,shapes,chains}

\title{Learn bifurcations of nonlinear parametric systems via equation-driven neural networks}

\author{
	Wenrui Hao\\
	Department of Mathematics\\
	Pennsylvania State University\\
	University Park, PA 16802, USA \\
	\texttt{wxh64@psu.edu} \\
	\And
	Chunyue Zheng \\
	Department of Mathematics\\
	Pennsylvania State University\\
	University Park, PA 16802, USA \\
	\texttt{cmz5199@psu.edu} \\
}

\begin{document}
\maketitle

\begin{abstract}
Nonlinear parametric systems have been widely used in modeling nonlinear dynamics in science and engineering. Bifurcation analysis of these nonlinear systems on the parameter space are usually used to study the solution structure such as the number of solutions and the stability.  In this paper, we develop a new machine learning approach to compute the bifurcations via so-called equation-driven neural networks (EDNNs). The EDNNs consist of a two-step optimization: the first step is to approximate the solution function of the parameter by training empirical solution data; the second step is to compute bifurcations by using the approximated neural network obtained in the first step. Both theoretical convergence analysis and numerical implementation on several examples have been performed to demonstrate the feasibility of the proposed method.
\end{abstract}

\keywords{neural networks \and nonlinear parametric systems \and bifurcations }

\section{Introduction}
Systems of nonlinear equations have played important roles in modeling natural phenomena from biology, physics, and materials science \cite{HNS,HHHS}. The solution structures of these nonlinear systems such as
bifurcations and multiple solutions \cite{HHHS,HSZ} are essential to understand  the nonlinear models. More specifically, the relationship between solutions and
parameters is the central question. In order to answer this question,  numerically computing bifurcation often requires large-scale computation, especially for high dimensional parameters. Thus,  efficient numerical algorithms for computing bifurcations of nonlinear parametric systems are keys to exploring
solution configurations, instability, and multiple
solutions \cite{bates2013numerically}. 

There are many methods developed for computing bifurcations of nonlinear systems. One type of these methods is the so-called {\em deflation} technique \cite{leykin2006newton}. It constructs a new augmented nonlinear system to convert the singular solution of the original system to a regular solution of the new system which can be computed by Newton's method.  But the augmented system introduces new variables and new equations and normally doubles the size of the original nonlinear system so it is hard to be applied to large-scale systems. Another direction in this research area is based on {\em homotopy continuation method} \cite{bates2018paramotopy}.  Several adaptive homotopy tracking methods \cite{hao2020adaptive,hao2021adaptive} have been developed to speed up the computation of finding bifurcation points. Some computational packages such as AUTO \cite{doedel1981auto}
and MATCONT \cite{dhooge2003matcont} have also been developed to study parameterized differential equations.  However, the homotopy continuation approach is based on one-dimensional parameter space and becomes inefficient and complicated when the parameter space is high dimensional.

Due to overcoming the curse of dimensionality, neural network techniques have been used to study the solution structure of nonlinear systems \cite{bernal2020machine,mourrain2006determining,huang2001neural,huang2002constrained}. Most of them focus on predicting the number of real solutions of polynomial systems by using feed-forward neural networks and learning the real discriminant locus of parameterized polynomial
equations by the supervised classification. Neural networks have also been applied to approximate any nonlinear continuous functional\cite{chen1993approximations}. Moreover, a universal approximation theorem is proved to show the possibility of neural networks in learning nonlinear operators from data\cite{chen1995universal}. Recently, deep operator networks are also proposed to realize the theorem in practice \cite{lu2019deeponet}. Neural networks have also been developed to solve nonlinear differential equations \cite{raissi2019physics,xu2020finite,gu2021selectnet} by providing a mesh free approach. Moreover, neural networks combine both fitting observation data and calibrating nonlinear differential equations together to solve mathematical models with unknown parameters \cite{kharazmi2021identifiability}. However, if there are singularities in the parameter space, the solutions of nonlinear models can be complex and not unique. In this case, the existing neural network approaches are not applicable anymore. In order to compute singularities/bifurcations of parameter space,
%
our paper focuses on developing an equation-driven neural network (EDNN) to learn the solution path by empirical solution data and further to compute the bifurcation points on the parameter space. 
The rest of the paper is organized as follows: Section 2 provides the formulation of EDNNs and learning algorithms;
Section 3 provides the convergence analysis of the proposed algorithms for computing bifurcation points;
Section 4 applies the proposed approach to several examples;  Finally, a conclusion is provided in Section 5.

\section{EDNN and learning algorithms}
Generally speaking, a nonlinear parametric system is written as $\mathbf{F}:
\mathbb{R}^n\times\mathbb{R}^d\rightarrow\mathbb{R}^n,$
\begin{equation}\label{Sys}
	\mathbf{F}(\mathbf{u},\mathbf{p})=\mathbf{0},
\end{equation}
where $\mathbf{p}$ is a parameter and $\mathbf{u}$ is the variable vector that depends on the parameter $\mathbf{p}$, i.e., $\mathbf{u}=\mathbf{u}(\mathbf{p})$. Suppose we have a solution at the starting point, namely $\mathbf{u}(\mathbf{p}_0)=\mathbf{u}_0$, various homotopy tracking algorithms can be used to compute the solution path, $\mathbf{u}(\mathbf{p})$. If $\mathbf{F}_\mathbf{u}(\mathbf{u},\mathbf{p})$ is nonsingular, the solution path $\mathbf{u}(\mathbf{p})$ is smooth and unique. However, when $\mathbf{F}_\mathbf{u}(\mathbf{u},\mathbf{p})$ becomes singular, the solution path hits the singularity and has different types of bifurcations \cite{bates2013numerically}. Especially for a high dimensional parameter space, computing the singularities is quite challenging. In order to address this challenge, we employ neural networks to approximate the solution on the parameter space. More specifically, fully connected neural networks, consisting of a series of fully connected layers shown in Fig. \ref{Fig:NN}, are functions from the input $p\in\mathbb{R}^d$ to the output $u\in\mathbb{R}^m$.  Then a neural network with $L$ hidden layers can be written as follows
\begin{eqnarray}
	\label{FCNN}
	\mathbf{u}_N(\mathbf{p};\theta) &=& W_Lh_{L-1}+b_L,\nonumber\\ ~h_i&=&\sigma(W_i h_{i-1}+b_i),~ i\in\{1,\cdots, L-1\}, \nonumber\\ \hbox{and }h_0&=&\mathbf{p},
\end{eqnarray} 
where $W_i\in\mathbb{R}^{d_{i}\times d_{i-1}}$ is the weight, $b_i\in\mathbb{R}^{d_{i}}$ is the bias, $d_i$ is the width of $i$-th hidden layer, and $\sigma$ is the	activation function (for example, the rectified linear unit (ReLU) or the sigmoid activation functions \cite{goodfellow2016deep}). By denoting all the parameters of the neural network, namely the weights and bias, as $\theta$, the neural network, $\mathbf{u}_N(\mathbf{p};\theta)$, is then trained by a collection of solution data points $\{\mathbf{p}_i,\mathbf{u}_i\}_{i  = 1}^{K}$ by solving $\mathbf{F}(\mathbf{u}_i,\mathbf{p}_i)=\mathbf{0}$. Then we have
$\displaystyle \theta=\mathop{\arg\min}_{\theta} f_1(\theta)$, where
\begin{eqnarray}
	f_1(\theta)&=&\sum_{i=1}^K  \|{\mathbf{u}}_N(\mathbf{p}_i,\theta) - \mathbf{u}_i(\mathbf{p}_i)\|^2 	+ \frac{\lambda}{2}  \| \mathbf{F}({\mathbf{u}}_N(\mathbf{p}_i,\theta) ,\mathbf{p}_i)\|^2,\quad \quad \label{OPt}
\end{eqnarray}  where
$\mathbf{p}_i$ and $\mathbf{u}_i$ represent data points,  $K$ is the number of data
points along the solution path. In this training process, the EDNN is formed by combining both the solution path fitting and the equation information together.

\begin{figure}[!ht]\centering
	\includegraphics[width=3in]{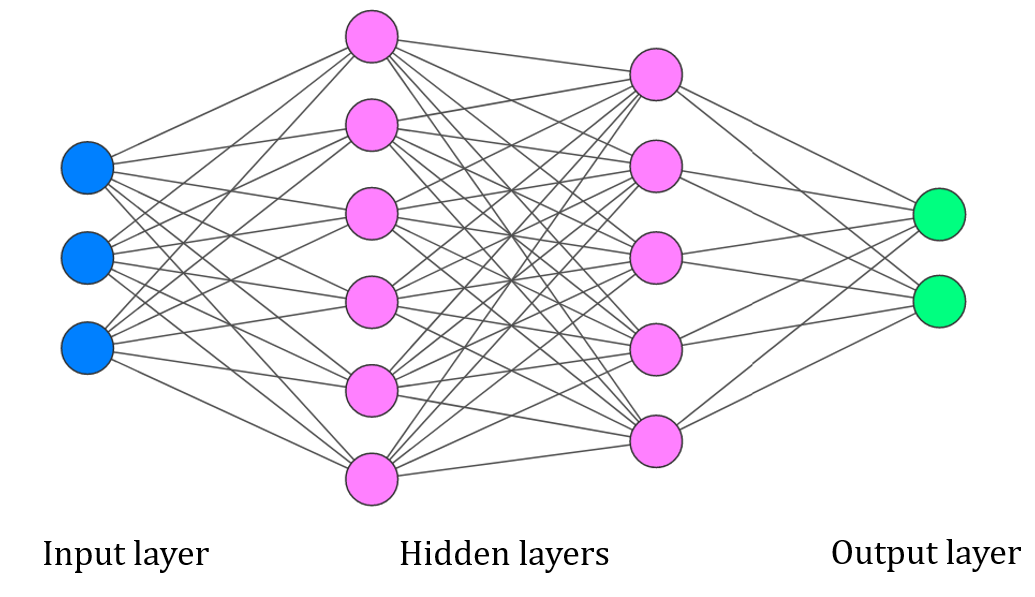}
	\caption{A regular 2-hidden-layer fully connected neural network. }\label{Fig:NN}
\end{figure}

Once the EDNN ${\mathbf{u}}_N(\mathbf{p},\theta)$ is trained, we compute the bifurcation point by solving
$\displaystyle (\mathbf{p}^*,\mathbf{v}^*)=\min_{\mathbf{p},\mathbf{v}}f_2(\mathbf{p},\mathbf{v})$, where

\begin{equation}
	\begin{aligned}
		f_2(\mathbf{p},\mathbf{v}) =	&\frac{\mathbf{v}^T\mathbf{F}^T_\mathbf{u}({\mathbf{u}}_N(\mathbf{p},\theta),\mathbf{p})\mathbf{F}_\mathbf{u}({\mathbf{u}}_N(\mathbf{p},\theta),\mathbf{p})\mathbf{v}}{\mathbf{v}^T\mathbf{v}}\\
		&+ \frac{\lambda}{2} \|\mathbf{F}({\mathbf{u}}_N(\mathbf{p},\theta)),\mathbf{p}) \|^2.
	\end{aligned}
	\label{eqn:optimize-bifur}
\end{equation}
Here $\mathbf{F}^T_\mathbf{u}(\mathbf{u},\mathbf{p})$ is the Jacobian matrix of $\mathbf{F}(\mathbf{u},\mathbf{p})$. Eq. (\ref{eqn:optimize-bifur}) combines both the nonlinear system and the eigenvector $\mathbf{v}$ corresponding to the zero eigenvalue of the Jacobian matrix.

\begin{algorithm}[H]
	\caption{The pseudocode of computing the bifurcation points by using the EDNN.}\label{alg:DNN-Bifur}
	\begin{algorithmic}
		\STATE \textbf{Input: }{The number of epochs $n$ and  solution data on one solution path $\{\mathbf{p}_i,\mathbf{u}_i\}_{i=1}^K$}.
		\STATE \textbf{Output: }{Bifurcations on the solution path $\mathbf{p}^*$}.
		\FOR {$i=1:n$}		
		\STATE Train the EDNN $\mathbf{u}_N(\mathbf{p},\theta)$ by solving the optimization problem in  \eqref{OPt} on each epoch;
		\ENDFOR
		\STATE Solve optimization problem \eqref{eqn:optimize-bifur} to get $\mathbf{p}^*$ with the trained EDNN $\mathbf{u}_N(\mathbf{p},\theta^*)$.
	\end{algorithmic}
\end{algorithm}

\section{Convergence Analysis}
In this section, we discuss the convergence analysis of Algorithm~\ref{alg:DNN-Bifur}. First, we explore the approximation rate of using neural network to approximate the solution path $\mathbf{u}(\mathbf{p})$ by assuming the continuity of $\mathbf{u}(\mathbf{p})$ only.

For a general $E\subset \mathbb{R}^d$ and for any $r\ge0$, the modulus of continuity of $f\in C(E)$ is defined as
\begin{equation}
	\begin{aligned}
		\omega_f^E(r):=\sup\{|f(\mathbf{x}) - f(\mathbf{y})|: &\mathbf{x},\mathbf{y} \in E,\\
		&\|\mathbf{x} - \mathbf{y}\| \le r \}.
	\end{aligned}
\end{equation}
In particular, $\omega_f(\cdot)$ is short of $\omega_f^E(\cdot)$ in the case of $E = [0,1]^d$.  Then an approximation rate by using the ReLU as activation function is established \cite{shen2021optimal}.
\begin{lemma}
	Given any bounded continuous function $f\in C(E)$ with $E\subset [-R,R]^d$ and $R > 0$, for any $N\in \mathbb{N}^+$, $L\in \mathbb{N}^+$, and $p\in[1,\infty]$, there exists a function $\varphi$ implemented by a ReLU network with width $C_1\max\{d\lfloor N^{1/d}\rfloor,N+2\}$ and depth $11L+C_2$ such that
	\begin{equation}
		\begin{aligned}
			&\|f - \varphi\|_{L^p(E)} \\
			\le& 131(2R)^{d/p}\sqrt{d}\omega_f^E(2R(N^2L^2\log_3(N+2))^{-1/d}),
		\end{aligned}
	\end{equation}
	where $C_1=16$ and $C_2 = 18$ if $p\in[1,\infty)$; $C_1 = 3^{d+3}$ and $C_2 = 18+2d$ if $p = \infty$.
	\label{thm:convergence-rate-u}
\end{lemma}

Then we generalize the result and obtain the estimate for  $\|\mathbf{u}- \mathbf{u}_N\|_{L^p(E)}$  as follows:
\begin{theorem}
	Given any bounded continuous function $\mathbf{u}:E\to\mathbb{R}^n$ with $E\subset [-R,R]^d$ and $R > 0$, for any $N\in \mathbb{N}^+$, $L\in \mathbb{N}^+$, and $p\in[1,\infty]$, there exists a function $\mathbf{u}_N$ implemented by a ReLU network with width $nC_1\max\{d\lfloor N^{1/d}\rfloor,N+2\}$ and depth $11L+C_2$ such that
	\begin{equation}
		\begin{aligned}
			&\|\mathbf{u}- \mathbf{u}_N\|_{L^p(E)} 
			\le C \omega_{\mathbf{u}}^E(2R(N^2L^2\log_3(N+2))^{-\frac{1}{d}}),
		\end{aligned}
	\end{equation}
	where 	
	$C=131(2R)^{\frac{d}{p}}n^{\frac{1}{p}}\sqrt{d}$
	and
	$	\displaystyle	\omega_{\mathbf{u}}^E(r):=	\max_{i}\omega_{u_i}^E(r) 
	$.
	\label{thm:convergence-rate-u-vector}
\end{theorem}

\begin{proof}
	By denoting $\mathbf{u} = [u_1,u_2,\cdots,u_n]$,	
	by Lemma~\ref{thm:convergence-rate-u},  each element $u_i$ is approximated by a ReLU network $\phi_i$ with width $C_1\max\{d\lfloor N^{1/d}\rfloor,N+2\}$ and depth $11L+C_2$ such that
	\begin{equation}
		\begin{aligned}
			&\|u_i - \phi_i\|_{L^p(E)} \\
			\le& 131(2R)^{\frac{d}{p}}\sqrt{d}\omega_{u_i}^E(2R(N^2L^2\log_3(N+2))^{-\frac{1}{d}}),
		\end{aligned}
	\end{equation}
	Next we construct $\mathbf{u}_N$ by stacking $\phi_i$ vertically and have the width as $nC_1\max\{d\lfloor N^{1/d}\rfloor,N+2\}$ and the depth as $11L+C_2$. Thus we have:
	\begin{equation}
		\begin{aligned}
			&\|\mathbf{u} - \mathbf{u}_N\|_{L^p(E)} = (\int_E\sum_{i=1}^n|u_i - \phi_i|^pd\mu)^{1/p}\\
			\le&(n\max_{i}\|u_i - \phi_i\|^p_{L^p(E)})^{\frac{1}{p}}\\
			\le& 131(2R)^{\frac{d}{p}}n^{\frac{1}{p}}\sqrt{d}\omega_{\mathbf{u}}^E(2R(N^2L^2\log_3(N+2))^{-\frac{1}{d}}).
		\end{aligned}
	\end{equation}
\end{proof}

Now we have the approximation rate of $\mathbf{u}$ by $\mathbf{u}_N$. Next we proceed to obtain the error estimate for the bifurcation approximation and first 
define $f(\mathbf{x},\mathbf{p},\mathbf{v})$ as follows:
\begin{equation}
	\begin{aligned}
		f(\mathbf{u},\mathbf{p},\mathbf{v}) &:= \frac{\mathbf{v}^T\mathbf{F}^T_\mathbf{u}(\mathbf{u},\mathbf{p})\mathbf{F}_\mathbf{u}(\mathbf{u},\mathbf{p})\mathbf{v}}{\mathbf{v}^T\mathbf{v}}+\frac{1}{2}\lambda \|\mathbf{F}(\mathbf{u},\mathbf{p}) \|^2,		
	\end{aligned}
	\label{eqn:loss_function}
\end{equation}
which recovers the minimization problem \eqref{eqn:optimize-bifur} by fixing
$\mathbf{u} = {\mathbf{u}}_N(\mathbf{p},\theta)$. Then we have the following theorem for
$\|\mathbf{p}^*-\mathbf{p}_N^* \|$  in terms of $\|\mathbf{u}^*-\mathbf{u}_N^* \|$.

\begin{theorem}
	Assume that $(\mathbf{u}(\mathbf{p}^*),\mathbf{p}^*,\mathbf{v}^*)$ is the exact solution to the minimization problem (\ref{eqn:loss_function})  with $\|\mathbf{v}^*\| = 1$, and $(\mathbf{u}_N(\mathbf{p}_N^*),\mathbf{p}_N^*,\mathbf{v}_N^*)$ is the solution to the minimization problem (\ref{eqn:optimize-bifur}) with $\|\mathbf{v}_N^*\| = 1$. If $\frac{d \mathbf{F}_\mathbf{u}}{d\mathbf{p}}(\mathbf{u}(\mathbf{p}^*),\mathbf{p}^*)\mathbf{v}^*$ is not in the range space of $\mathbf{F}_\mathbf{u}(\mathbf{u}(\mathbf{p}^*),\mathbf{p}^*)$, then the following estimate for $\|\mathbf{p}^*-\mathbf{p}_N^* \|$ holds:
	\begin{equation}
		\begin{aligned}
			\|\mathbf{p}_N^*-\mathbf{p}^*\|^2\le& C_1f(\mathbf{u}_N(\mathbf{p}_N^*)),\mathbf{p}_N^*,\mathbf{v}_N^*)  + C_2\|\mathbf{u}_N(\mathbf{p}_N^*)-\mathbf{u}(\mathbf{p}_N^*)\|
		\end{aligned}
	\end{equation}
	\label{thm:convergence-rate-p}
\end{theorem}

\begin{proof}
	By denoting
	\begin{equation}
		h(\mathbf{p},\mathbf{v}) := f(\mathbf{u}(\mathbf{p}),\mathbf{p},\mathbf{v}),
	\end{equation}
	we have 
	\begin{equation}
		h(\mathbf{p}^*,\mathbf{v}^*) = 0, \quad \nabla_{\mathbf{p},\mathbf{v}} h(\mathbf{p}^*,\mathbf{v}^*) = \mathbf{0}.
	\end{equation}
	By Taylor expansion, we obtain
	
	\begin{equation}
		\begin{aligned}
			&f(\mathbf{u}_N(\mathbf{p}_N^*),\mathbf{p}_N^*,\mathbf{v}_N^*) \\
			=&f(\mathbf{u}_N(\mathbf{p}_N^*)-\mathbf{u}(\mathbf{p}_N^*)+\mathbf{u}(\mathbf{p}_N^*),\mathbf{p}_N^*,\mathbf{v}_N^*) \\
			=&f(\mathbf{u}(\mathbf{p}_N^*),\mathbf{p}_N^*,\mathbf{v}_N^*) +  f_{\mathbf{u}}(\xi_1,\mathbf{p}_N^*,\mathbf{v}_N^*)(\mathbf{u}_N(\mathbf{p}_N^*)-\mathbf{u}(\mathbf{p}_N^*))\\
			=&h(\mathbf{p}_N^*,\mathbf{v}_N^*) +  f_{\mathbf{u}}(\xi_1,\mathbf{p}_N^*,\mathbf{v}_N^*)(\mathbf{u}_N(\mathbf{p}_N^*)-\mathbf{u}(\mathbf{p}_N^*)),
		\end{aligned}
	\end{equation}
	where $\xi_1 = \mathbf{u}(\mathbf{p}_N^*) + \theta_1(\mathbf{u}_N(\mathbf{p}_N^*)-\mathbf{u}(\mathbf{p}_N^*))$ with some $\theta_1\in[0,1]$.
	By denoting		 $(\mathbf{p}_\xi,\mathbf{v}_\xi) = (\mathbf{p}^*,\mathbf{v}^*) + \theta_2(\mathbf{p}_N^*-\mathbf{p}^*,\mathbf{v}_N^*-\mathbf{v}^*)$ with $\theta_2\in[0,1]$, and $\Delta U = [\mathbf{p}_N^*-\mathbf{p}^*;\mathbf{v}_N^*-\mathbf{v}^*]$, we derive
	
	\begin{equation}
		\begin{aligned}
			&f(\mathbf{u}_N(\mathbf{p}_N^*),\mathbf{p}_N^*,\mathbf{v}_N^*) \\
			=& h(\mathbf{p}^*,\mathbf{v}^*)+   \nabla_{\mathbf{p},\mathbf{v}} h(\mathbf{p}^*,\mathbf{v}^*)[\mathbf{p}_N^*-\mathbf{p}^*;\mathbf{v}_N^*-\mathbf{v}^*]\\
			&+\frac{1}{2}\Delta U ^T\nabla_{\mathbf{p},\mathbf{v}}^2 h(\mathbf{p}_\xi,\mathbf{v}_\xi)\Delta U 
			+    f_{\mathbf{u}}(\xi_1,\mathbf{p}_N^*,\mathbf{v}_N^*)(\mathbf{u}_N(\mathbf{p}_N^*)-\mathbf{u}(\mathbf{p}_N^*))\\
			=&\frac{1}{2}\Delta U ^T\nabla_{\mathbf{p},\mathbf{v}}^2 h(\mathbf{p}_\xi,\mathbf{v}_\xi)\Delta U +   f_{\mathbf{u}}(\xi_1,\mathbf{p}_N^*,\mathbf{v}_N^*)(\mathbf{u}_N(\mathbf{p}_N^*)-\mathbf{u}(\mathbf{p}_N^*)),
		\end{aligned}
	\end{equation}
	which implies 
	\begin{equation}
		\begin{aligned}
			&\|\frac{1}{2}\Delta U ^T\nabla_{\mathbf{p},\mathbf{v}}^2 h(\mathbf{p}_\xi,\mathbf{v}_\xi)\Delta U \|\\
			=& \|f(\mathbf{u}_N(\mathbf{p}_N^*),\mathbf{p}_N^*,\mathbf{v}_N^*)  - f_{\mathbf{u}}(\xi_1,\mathbf{p}_N^*,\mathbf{v}_N^*)(\mathbf{u}_N(\mathbf{p}_N^*)-\mathbf{u}(\mathbf{p}_N^*))\|\\
			\le & \|f_{\mathbf{u}}(\xi_1,\mathbf{p}^*,\mathbf{v}_N^*)\|\|\mathbf{u}_N(\mathbf{p}_N^*)-\mathbf{u}(\mathbf{p}_N^*)\|+f(\mathbf{u}_N(\mathbf{p}_N^*),\mathbf{p}_N^*,\mathbf{v}_N^*) \\
			\le &f(\mathbf{u}_N(\mathbf{p}_N^*),\mathbf{p}_N^*,\mathbf{v}_N^*)  + C_0\|\mathbf{u}_N(\mathbf{p}_N^*)-\mathbf{u}(\mathbf{p}_N^*)\|.
		\end{aligned}
		\label{eqn:get-bound}
	\end{equation}
	If $\nabla_{\mathbf{p},\mathbf{v}}^2h(\mathbf{p}^*,\mathbf{v}^*)$ is positive definite, then we are able to get the following estimate on $\|\mathbf{p}_N^*-\mathbf{p}^*\|$ based on \eqref{eqn:get-bound}:
	\begin{equation}
		\begin{aligned}
			&C\|[\mathbf{p}_N^*-\mathbf{p}^*;\mathbf{v}_N^*-\mathbf{v}^*]\|^2 	\le &f(\mathbf{u}_N(\mathbf{p}_N^*),\mathbf{p}_N^*,\mathbf{v}_N^*)  + C_0\|\mathbf{u}_N(\mathbf{p}_N^*)-\mathbf{u}(\mathbf{p}_N^*)\|,\\
		\end{aligned}
	\end{equation}
	which is equivalent to	
	\begin{equation}
		\begin{aligned}
			\|\mathbf{p}_N^*-\mathbf{p}^*\|^2\le &C_1 f(\mathbf{u}_N(\mathbf{p}_N^*),\mathbf{p}_N^*,\mathbf{v}_N^*) + C_2\|\mathbf{u}_N(\mathbf{p}_N^*)-\mathbf{u}(\mathbf{p}_N^*)\|.
		\end{aligned}
	\end{equation}

	Now we proceed to prove $\nabla_{\mathbf{p},\mathbf{v}}^2h(\mathbf{p}^*,\mathbf{v}^*)$ is positive definite.
	By denoting $A=\mathbf{F}^T_\mathbf{u}(\mathbf{u},\mathbf{p})\mathbf{F}_\mathbf{u}(\mathbf{u},\mathbf{p})$ in \eqref{eqn:loss_function}, we get
	%
	\begin{equation}
		\begin{aligned}
			\frac{\partial h}{\partial \mathbf{p}_k}=&  
			\frac{\mathbf{v}^T}{\|\mathbf{v}\|} \frac{dA}{d\mathbf{p}_k}\frac{\mathbf{v}}{\|\mathbf{v}\|}+\lambda\sum_{i} \mathbf{F}_{i}\frac{d\mathbf{F}_{i}}{d \mathbf{p}_k}
		\end{aligned}
	\end{equation}
	and
	\begin{equation}
		\begin{aligned}
			\frac{\partial h}{\partial \mathbf{v}_k}
			&=\frac{(2\sum_{i}A_{ik}\mathbf{v}_i)(\sum_{i} \mathbf{v}^2_i)-2\mathbf{v}_k \mathbf{v}^TA\mathbf{v} }{(\sum_{i} \mathbf{v}^2_i)^2}.
		\end{aligned}
	\end{equation}
	
	Hence
	\begin{equation}
		\begin{aligned}
			&\frac{\partial^2 h}{\partial \mathbf{p}_k\partial \mathbf{p}_r}(\mathbf{p}^*,\mathbf{v}^*)\\
			=&\mathbf{v}^{*T}\frac{d^2 A}{d\mathbf{p}_k d\mathbf{p}_r}\mathbf{v}^* +\lambda\sum_{i} \frac{d \mathbf{F}_{i}}{d \mathbf{p}_r}\frac{d \mathbf{F}_{i}}{d \mathbf{p}_k }+\lambda\sum_{i}  \mathbf{F}_{i}\frac{d^2 \mathbf{F}_{i}}{d \mathbf{p}_k d \mathbf{p}_r }\\
			=&\mathbf{v}^{*T} ( \frac{d \mathbf{F}^T_\mathbf{u}}{d\mathbf{p}_k}\frac{d \mathbf{F}_\mathbf{u}}{d\mathbf{p}_r}+\frac{d \mathbf{F}^T_\mathbf{u}}{d\mathbf{p}_r}\frac{d \mathbf{F}_\mathbf{u}}{d\mathbf{p}_k})\mathbf{v}^*
		\end{aligned}
	\end{equation}
	and
	\begin{equation}
		\begin{aligned}
			&\frac{\partial^2 h}{\partial \mathbf{v}_k\partial \mathbf{p}_r}(\mathbf{p}^*,\mathbf{v}^*) \\
			= &  \frac{(2\sum_{i}\frac{\partial A_{ik}}{\partial \mathbf{p}_r}\mathbf{v}_i)(\sum_{i} \mathbf{v}^2_i)-2\mathbf{v}_k \mathbf{v}^T\frac{\partial A}{\partial \mathbf{p}_r}\mathbf{v} }{(\sum_{i} \mathbf{v}^2_i)^2}\\
			=&2[\frac{d}{d\mathbf{p}_r} (\mathbf{F}^T_\mathbf{u}\mathbf{F}_\mathbf{u})]_{k,:}\mathbf{v}^*\frac{\partial^2 h}{\partial \mathbf{v}_k\partial \mathbf{v}_r}(\mathbf{p}^*,\mathbf{v}^*)
			=2[\mathbf{F}^T_\mathbf{u}\mathbf{F}_\mathbf{u}]_{k,r}.
		\end{aligned}
	\end{equation}

	Thus for $\mathbf{x}\in\mathbb{R}^d$ and $\mathbf{y}\in\mathbb{R}^n$, let $\mathbf{U} = [\mathbf{x};\mathbf{y}]$
	\begin{equation}
		\begin{aligned}
			&\mathbf{U} ^T\nabla^2 h(\mathbf{p}^*,\mathbf{v}^*)\mathbf{U} \\
			=&[\mathbf{x}^T\  \mathbf{y}^T]\begin{bmatrix}
				2\mathbf{v}^{*T}(\frac{d \mathbf{F}^T_\mathbf{u}}{d\mathbf{p}}\frac{d \mathbf{F}_\mathbf{u}}{d\mathbf{p}})\mathbf{v}^* &2\mathbf{v}^{*T}\frac{d \mathbf{F}^T_\mathbf{u}}{d\mathbf{p}}\mathbf{F}_\mathbf{u} \\
				2\mathbf{F}^T_\mathbf{u}\frac{d \mathbf{F}_\mathbf{u}}{d\mathbf{p}} \mathbf{v}^*&2\mathbf{F}_\mathbf{u}^T\mathbf{F}_\mathbf{u}
			\end{bmatrix}\begin{bmatrix}
				\mathbf{x}\\
				\mathbf{y}
			\end{bmatrix}\\
			=&2\mathbf{x}^T	\mathbf{v}^{*T}(\frac{d \mathbf{F}^T_\mathbf{u}}{d\mathbf{p}}\frac{d \mathbf{F}_\mathbf{u}}{d\mathbf{p}})\mathbf{v}^* \mathbf{x}
			+4\mathbf{y}^T\mathbf{F}^T_\mathbf{u}\frac{d \mathbf{F}_\mathbf{u}}{d\mathbf{p}} \mathbf{v}^*\mathbf{x} \\
			&+2\mathbf{y}^T\mathbf{F}_\mathbf{u}^T\mathbf{F}_\mathbf{u}\mathbf{y}
		\end{aligned}
		\label{eqn:positive_definite}
	\end{equation}
	
	where $	\mathbf{v}^{*T}(\frac{d \mathbf{F}^T_\mathbf{u}}{d\mathbf{p}}\frac{d \mathbf{F}_\mathbf{u}}{d\mathbf{p}})\mathbf{v}^* \in \mathbb{R}^{d\times d}$ is defined as
	$$
	(\mathbf{v}^{*T}(\frac{d \mathbf{F}^T_\mathbf{u}}{d\mathbf{p}}\frac{d \mathbf{F}_\mathbf{u}}{d\mathbf{p}})\mathbf{v}^* )_{ij} = \mathbf{v}^{*T}(\frac{d \mathbf{F}^T_\mathbf{u}}{d\mathbf{p}_i}\frac{d \mathbf{F}_\mathbf{u}}{d\mathbf{p}_j})\mathbf{v}^*
	$$
	and $(\frac{d \mathbf{F}_\mathbf{u}}{d\mathbf{p}})_i = \frac{d \mathbf{F}_\mathbf{u}}{d\mathbf{p}_i} \in \mathbb{R}^{n\times n}$.
	
	It is equivalent to prove the following matrix is non-singular
	\begin{equation}
		\begin{bmatrix}
			\mathbf{F}_\mathbf{u} & \frac{d \mathbf{F}_\mathbf{u}}{d\mathbf{p}}\mathbf{v}^*\\
			\mathbf{v}^{*T}&0
		\end{bmatrix}	\begin{bmatrix}
			\mathbf{a}\\
			\mathbf{b}
		\end{bmatrix} = \begin{bmatrix}
			0\\
			0
		\end{bmatrix}
	\end{equation}
	which is to say, if $\frac{d \mathbf{F}_\mathbf{u}}{d\mathbf{p}}(\mathbf{u}(\mathbf{p}^*),\mathbf{p}^*)\mathbf{v}^*$ is not in the range space of $\mathbf{F}_\mathbf{u}(\mathbf{u}(\mathbf{p}^*),\mathbf{p}^*)$. Thus by this assumption, we have $\nabla^2h(\mathbf{p}^*,\mathbf{v}^*)$ be positive definite.
\end{proof}

\begin{remark}
	Assuming $\mathbf{p}\subset [-R,R]^d$, an immediate consequence of Theorem~\ref{thm:convergence-rate-u-vector} is the following estimate:
	\begin{equation}
		\begin{aligned}
			&\|\mathbf{p}_N^*-\mathbf{p}^*\|^2\le C_1f(\mathbf{u}_N(\mathbf{p}_N^*)),\mathbf{p}_N^*,\mathbf{v}_N^*)  \\
			&+ C_2\sqrt{d}\omega_{\mathbf{u}}^E(2R(N^2L^2\log_3(N+2))^{-1/d})			
		\end{aligned}
	\end{equation}
	with a ReLU neural network with $\mathcal{O}(L)$ depth and $\mathcal{O}(N)$ width.
\end{remark}

\section{Numerical results}

In this section, we apply Algorithm~\ref{alg:DNN-Bifur} to approximate bifurcations for different nonlinear parametric systems. We use pytorch to implement the proposed numerical algorithms and use {\it torch.mean} to calculate the loss function, which is different with \eqref{OPt} and \eqref{eqn:optimize-bifur} by a constant multiplier.

\subsection{One-dimensional example}\label{sec:dnn-bifur-ex1}
First, we apply the EDNN to the following one dimensional parametric problem
\begin{equation}
	F(x,p) =	x^2 - p = 0,
	\label{eqn:ex1}
\end{equation}
where the bifurcation point $p^* = 0$ is a turning point. We generate the training data by  $\{(p_i,\sqrt{p_i})\}_{i=1}^{K = 1600}$  and randomly choose $p_i\in[0,2]$ with the uniform distribution. By choosing different number of hidden units for the one-hidden-layer EDNN, we  summarize the absolute errors of the EDNN approximation in Table~\ref{table:ex1} which confirms the numerical convergence as $N$ increases. Fig~\ref{fig:ex1loss1} shows the natural log of the loss v.s. epoches for both the training step and the approximating step which achieve $10^{-5}$ and $10^{-3}$ respectively.

\begin{table}[ht]
	\caption{Numerical errors of  approximating bifurcation by one-hidden-layer EDNNs with different width, $N$, for Eq. \eqref{eqn:ex1}.}
	\begin{center}
		\begin{tabular}{|c|c|}
			\hline
			Width & $|p^*-p^*_N|$  \\
			\hline
			N = 20 & 0.0970  \\
			\hline
			N = 40 &0.0528   \\
			\hline
			N = 80 &0.0237  \\
			\hline
			N = 160 &0.0153\\
			\hline
			N = 320 & 0.0100 \\
			\hline
		\end{tabular}
	\end{center}
	\label{table:ex1}
\end{table}

\begin{figure}[th]
	\centering
	\includegraphics[width=0.45\linewidth]{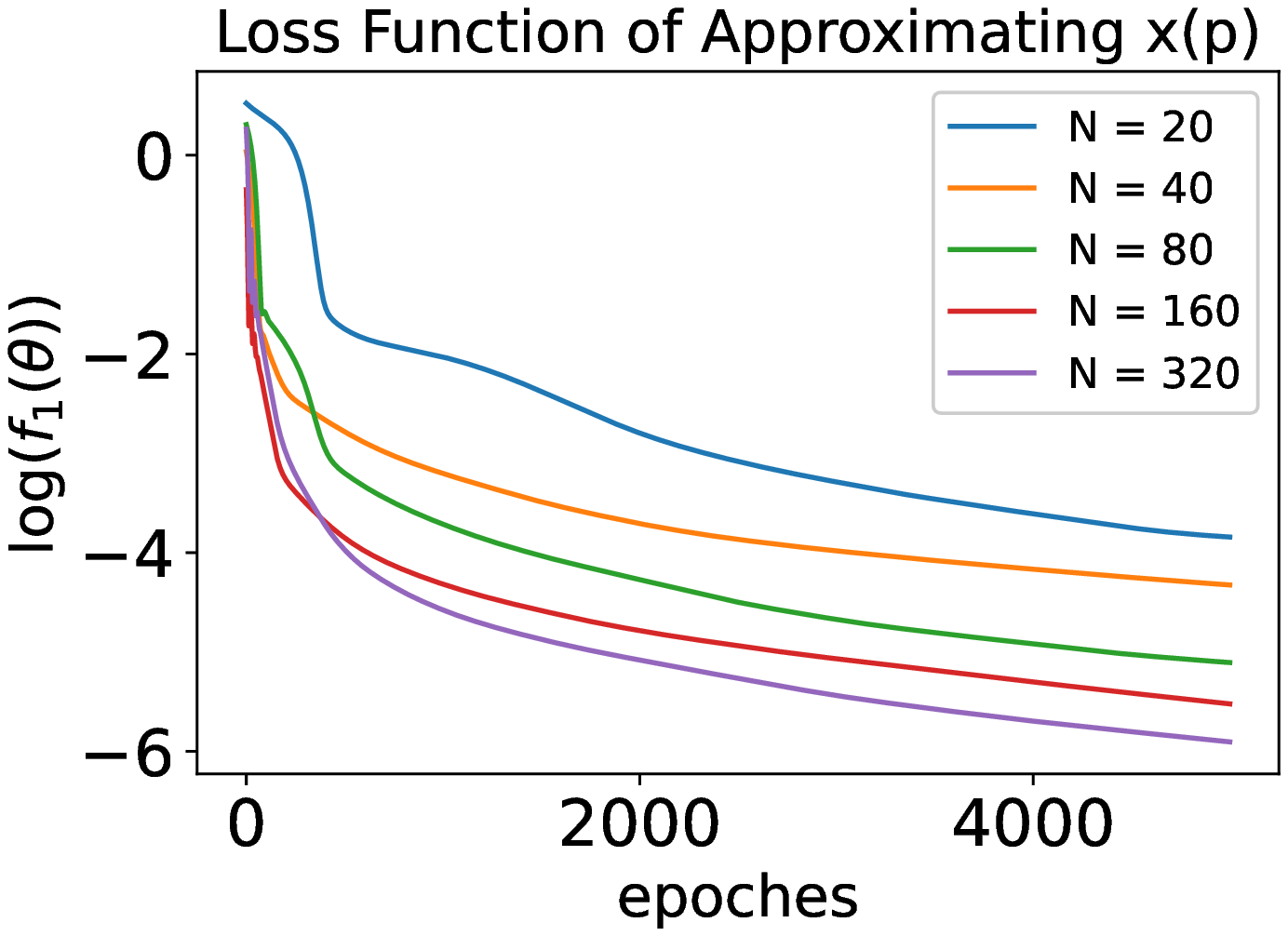}
	\includegraphics[width=0.45\linewidth]{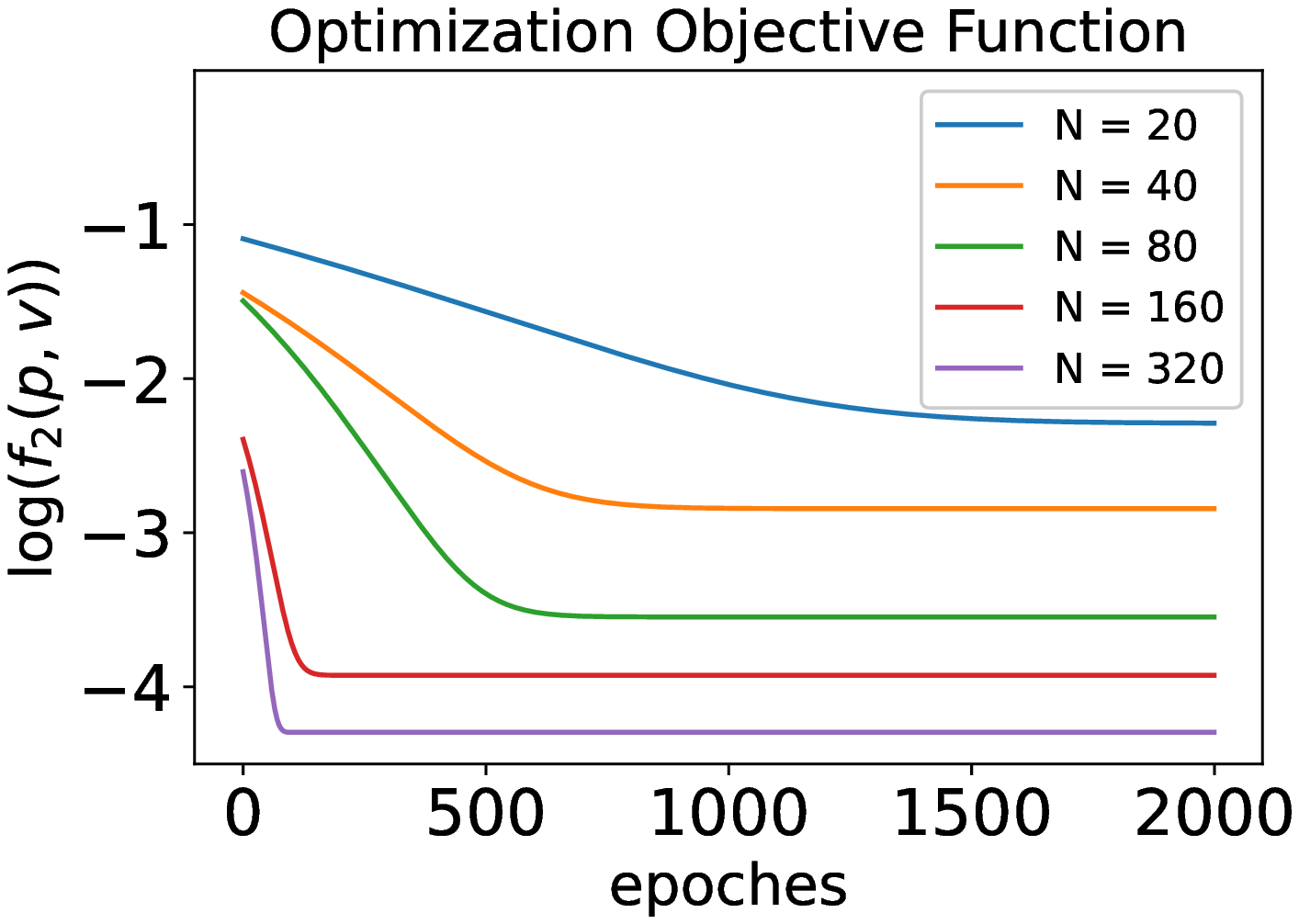}
	\caption{{\bf Upper:} $log(f_1(\theta))$ shown in (\ref{OPt}) v.s. epoches for training EDNNs to learn the solution path; {\bf Lower:} $log(f_2(\mathbf{p},\mathbf{v}))$ shown in (\ref{eqn:optimize-bifur}) v.s. epoches for computing the bifurcation point by EDNNs.}
	\label{fig:ex1loss1}
\end{figure}

\subsection{Two \& three dimensional examples}\label{sec:dnn-bifur-ex2}
We consider the following general polynomial equation
\begin{equation}
	F(x,\mathbf{p},n) = x^n+\mathbf{p}_{n-1}x^{n-1}+\cdots +  \mathbf{p}_{0}= 0.\label{poly}
\end{equation}
The
solution structure of the general polynomial equation (\ref{poly}) over the parameter space is an important question to explore. Normally, the discriminant has been used to answer this question by illustrating the boundary
between regions where the solution structure changes \cite{bernal2020machine}. The discriminant is a singularity/bifurcation point if other parameters are given. In this example, we explore the discriminant by using our algorithm for both two and three dimensional cases.

\subsubsection{Two dimensional case}
We write the quadratic polynomial as 
\begin{equation}
	F(x,b,c) = x^2+bx+c= 0.
	\label{eqn:n=2}
\end{equation}
Then the bifurcation appears at $c=\frac{b^2}{4}$. The training data is collected by $\{b_i,c_i,\frac{-b_i + \sqrt{b_i^2-4c_i}}{2}\}_{i=1}^{K = 25000}$. To generate the dataset, we first get 5000 random $b_i$ from $[-2,2]$ with a uniform distribution. For each $b_i$, $c_i$ is drawn from the uniform distribution on $[\frac{b_i^2}{8},\frac{b_i^2}{4}]$ for 5 times. The range for $c_i$ is chosen to guarantee the existence of real solutions. The bifurcation curves are plotted in Fig. \ref{fig:ex3_1} for different one-hidden-layer EDNNs and also compared with the analytical bifurcation curve, $c=\frac{b^2}{4}$.

\begin{figure}[th]
	\centering
	\includegraphics[width=0.45\linewidth]{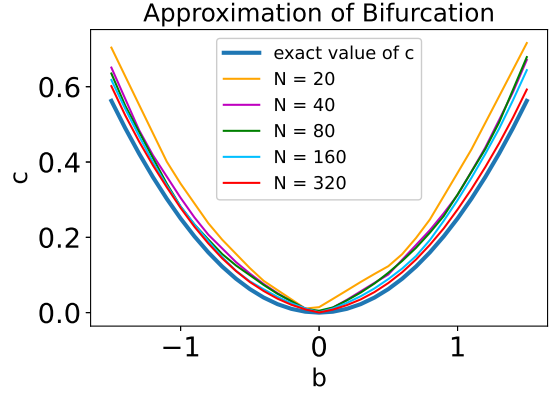}
	\caption{Approximation of the bifurcation curve for Eq. \eqref{eqn:n=2} with different one-hidden-layer EDNNs with $N$ nodes.}
	\label{fig:ex3_1}
\end{figure}

\begin{figure}[th]
	\centering
	\includegraphics[width=0.45\linewidth]{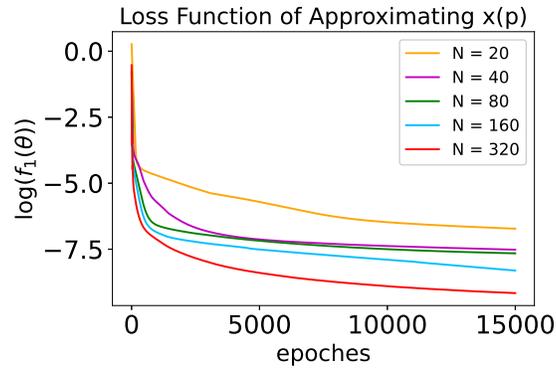}
	\caption{$log(f_1(\theta))$ shown in (\ref{OPt}) v.s. epoches for training EDNNs to learn the solution path for Eq. (\ref{eqn:n=2}).}
	\label{fig:ex3_2}
\end{figure}

\subsubsection{Three dimensional case}
The monic cubic polynomial has the following general form
\begin{equation}
	F(x,\mathbf{p}) = x^3+bx^2+cx + d= 0,
	\label{eqn:n=3}
\end{equation}
where $\mathbf{p}=(b,c,d)$ is the parameter vector.
By defining the discriminant  as 
\begin{equation}
	\Delta = (\frac{bc}{6} - \frac{b^3}{27} - \frac{d}{2})^2 + (\frac{c}{3} - \frac{b^2}{9})^3,
\end{equation}
we compute three roots of $F(x)$ as follows
\begin{equation}
	\begin{aligned}
		x_1 =& -\frac{b}{3} + \sqrt[3]{\frac{bc}{6} - \frac{b^3}{27} - \frac{d}{2} + \sqrt{\Delta} } + \sqrt[3]{\frac{bc}{6} - \frac{b^3}{27} - \frac{d}{2} - \sqrt{\Delta} }  \\
		x_2 =& -\frac{b}{3} + \frac{-1+\sqrt{3}i}{2}\sqrt[3]{\frac{bc}{6} - \frac{b^3}{27} - \frac{d}{2} + \sqrt{\Delta} } \\
		& + \frac{-1-\sqrt{3}i}{2}\sqrt[3]{\frac{bc}{6} - \frac{b^3}{27} - \frac{d}{2} - \sqrt{\Delta} }  \\
		x_3 =& -\frac{b}{3} + \frac{-1-\sqrt{3}i}{2}\sqrt[3]{\frac{bc}{6} - \frac{b^3}{27} - \frac{d}{2} + \sqrt{\Delta} } \\
		&+ \frac{-1+\sqrt{3}i}{2}\sqrt[3]{\frac{bc}{6} - \frac{b^3}{27} - \frac{d}{2} - \sqrt{\Delta} }.  \\
	\end{aligned}
\end{equation}

 When $\Delta = 0$, the equation has three real roots counting multiplicity.
	\begin{itemize}
		\item If $ (\frac{bc}{6} - \frac{b^3}{27} - \frac{d}{2})^2  = -(\frac{c}{3} - \frac{b^2}{9})^3\ne 0$, the equation has two distinct real roots of multiplicity 1 and 2, respectively.
		\item  If $ (\frac{bc}{6} - \frac{b^3}{27} - \frac{d}{2})^2  = -(\frac{c}{3} - \frac{b^2}{9})^3= 0$, the equation has one real root of multiplicity 3.
\end{itemize}

Then we have the bifurcation curve described by $(c,d) = (\frac{b^2}{3},\frac{b^3}{27})$.
The training data is collected by using $\{b_i,c_i,d_i, -\frac{b_i}{3} + 2\sqrt[3]{\frac{b_ic_i}{6} - \frac{b_i^3}{27} - \frac{d_i}{2}}\}_{i=1}^{K = 12000}$. To generate the training dataset, we first get 3000 random $b_i$ from the uniform distribution on $[0,2]$.  For each  $b_i$, $c_i$ is drawn from the uniform distribution on $[0,\frac{b_i^2}{3}]$ for 4 times. After obtaining pairs of $(b_i,c_i)$, we solve the equation $ (\frac{bc}{6} - \frac{b^3}{27} - \frac{d}{2})^2  = -(\frac{c}{3} - \frac{b^2}{9^2})^3= 0$ to get the corresponding value of $d_i$. We show the comparisons between the approximation of bifurcation curves by different one-hidden-layer EDNNs and the analytical one in Fig. \ref{fig:ex4_cd}. The numerical errors are computed 
\begin{equation}
	\int_{0}^{\frac{3}{2}}|c_N(b)-c(b)|db\hbox{~and~}\int_{0}^{\frac{3}{2}}|d_N(b)-d(b)|db \label{bc_err}
\end{equation}
and are shown in Table~\ref{table:ex4_cd} for different EDNNs to demonstrate the convergence with respect to $N$.

\begin{figure}[th]
	\centering
	\includegraphics[width=0.45\linewidth]{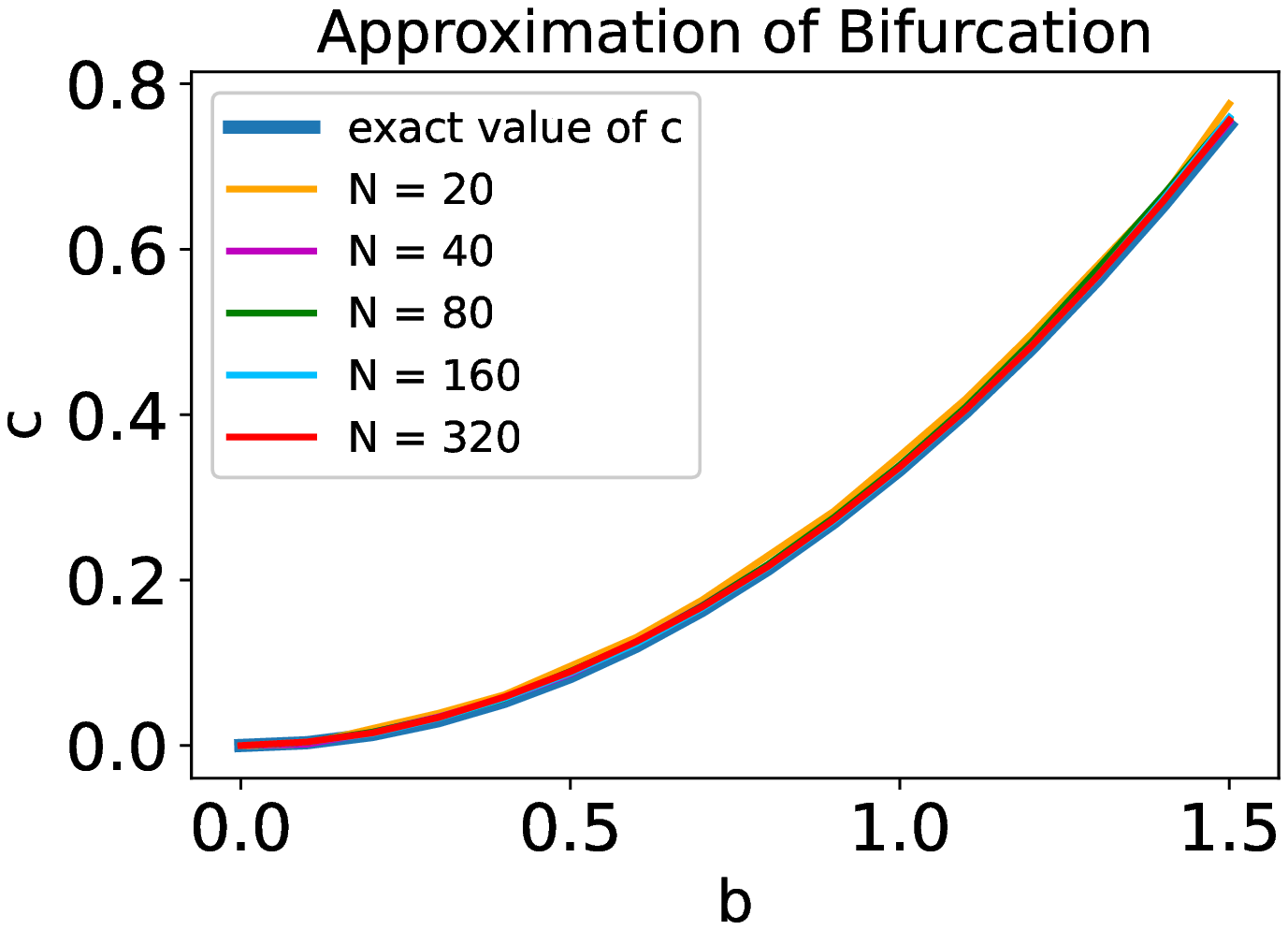}
	\includegraphics[width=0.45\linewidth]{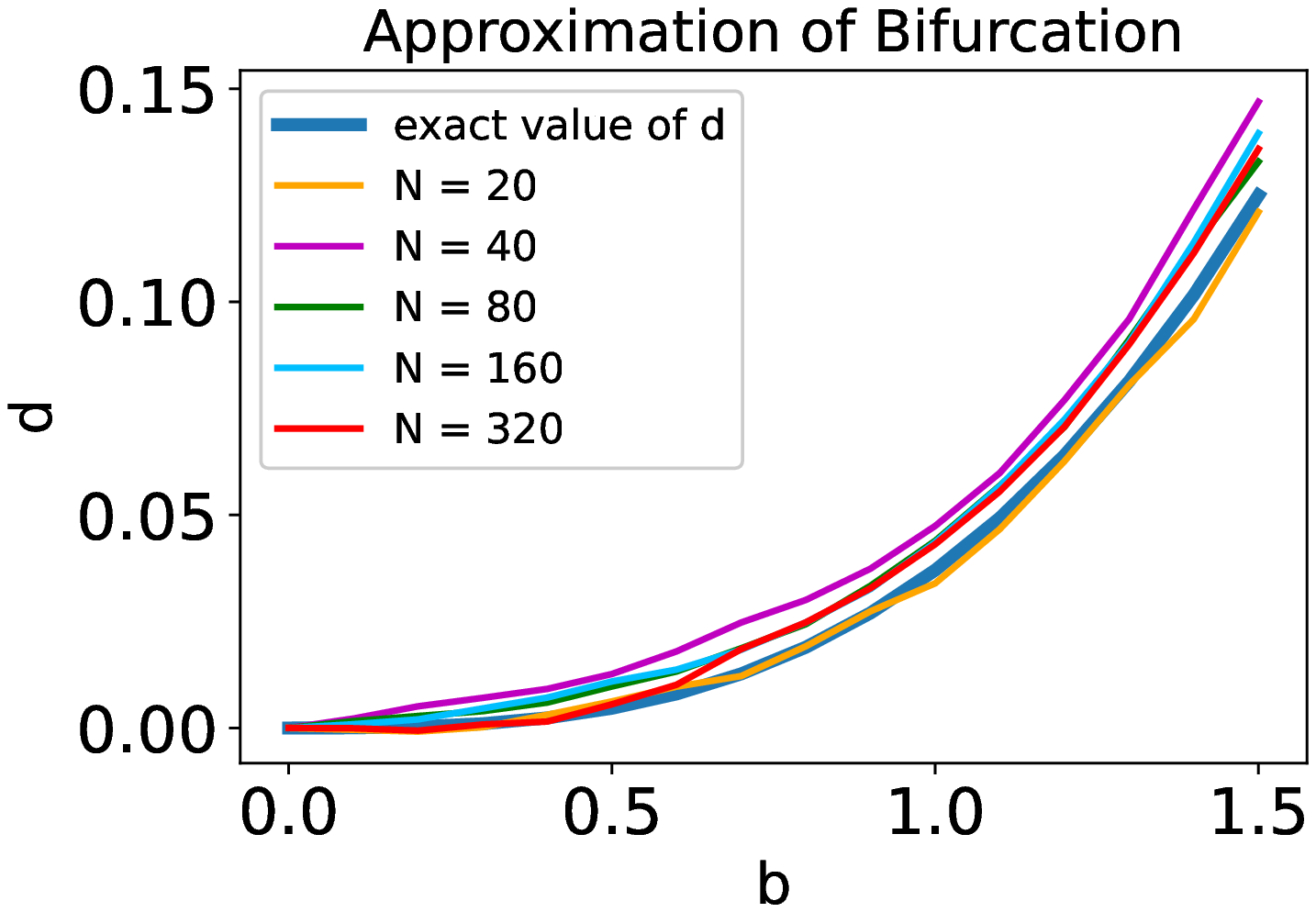}
	\caption{Comparison of the bifurcation curves obtained by different one-hidden-layer EDNNs with $N$ neurons and the exact bifurcation curves for both $c$ v.s. $b$ ({\bf Upper}) and $d$ v.s. $b$ ({\bf Lower}) where $b\in[0,3/2]$.}
	\label{fig:ex4_cd}
	
\end{figure}
\begin{figure}[th]
	\centering
	\includegraphics[width=0.45\linewidth]{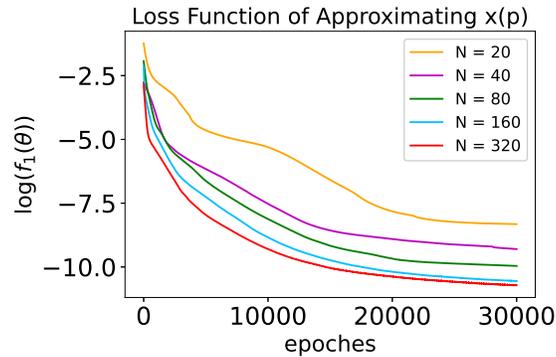}
	\caption{$log(f_1(\theta))$ shown in (\ref{OPt}) v.s. epoches for training EDNNs to learn the solution path.}
	\label{fig:ex4_loss}
\end{figure}

\begin{table}[t]
	{		\begin{center}
			\begin{tabular}{|l|c|c|} \hline
				N&Numerical error for $c$ &Numerical error for $d$\\\hline	
				20 & 0.0183& 0.0023 \\
				\hline				
				40 &  0.0060& 0.0151\\\hline	
				80 &  0.0087 &  0.0085\\\hline	
				160 &  0.0048 & 0.0091\\\hline	
				320 &   0.0055&  0.0066 \\
				\hline				
			\end{tabular}
	\end{center}}
	\caption{The numerical errors defined in Eq. (\ref{bc_err}) for different one-hidden-layer EDNNs with $N$ neurons.}
	\label{table:ex4_cd}
\end{table}

\subsection{Nonlinear parametric boundary value problem}
\label{sec:ex3}
We consider the following 1D nonlinear boundary value problem.
\begin{equation}\left\{
	\begin{aligned}
		&u_{xx} = u^2(u^2-p),\\
		&u_x(0) = 0,	u(1) = 0,
	\end{aligned}\right.\label{ex3}
\end{equation}
where $p$ is a parameter. There are multiple solutions for any given parameter $p$, and the solution structure gets complex when the parameter $p$ is increased\cite{hao2011domain}. As we track along solution paths with respect to $p$, turning points occur and introduces more solutions. In order to compute the bifurcation point $p^*$, we  use one-hidden-layer EDNNs with the sigmoid activation function and $N$ neurons. 
We discretize (\ref{ex3}) by using finite difference method with the stepsize $h=0.2$.  Since there are four solution branches shown in Fig~\ref{fig:ex3},  we use 7600 points on each solution branch to train five different EDNNs and choose the best one to compute the bifurcation points. The numerical errors are summarized in Table~\ref{table:ex3}.

\begin{table*}[ht]
	{		\begin{center}
			\resizebox{\textwidth}{!}
			{			\begin{tabular}{|l|c|c|c|c|} \hline
					$N$&
					Branch 1($\lambda=20$)& Branch 2($\lambda=100$) &Branch 3($\lambda=56$)&Branch 4($\lambda=0.3$)\\\hline
					10 &  0.1811 & 0.0163& 0.0272  & 0.0763\\\hline
					50 &0.0869&  0.0038&  0.0320 &  0.0242 \\\hline
					100  &0.0446&0.0041&  0.0003&   0.0421\\\hline
					200   & 0.0122&0.0218 &  0.0025 & 0.0349\\\hline
					500   & 0.0408 & 0.0195 &  0.0116 &   0.0287\\\hline											
			\end{tabular}}
	\end{center}}
	\caption{Numerical errors of bifurcation points, $|p_N^* - p^*|$,  for different one-hidden-layer EDNNs with $N$ neurons on different solution branches shown in Fig. \ref{fig:ex3} with different $\lambda$ used in Eq. (\ref{OPt}).}
	\label{table:ex3}
\end{table*}

\begin{figure}
	\centering
	\includegraphics[width=0.45\linewidth]{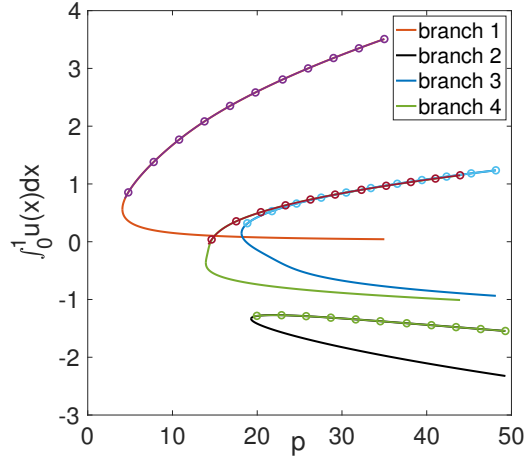}
	\caption{The solution structure v.s. $p$ for Eq. (\ref{ex3}). Here the solid curves represent solution paths computed by finite difference method, while the colored ones with circles correspond to the solution of EDNNs with hidden unit $N = 500$.}
	\label{fig:ex3}
\end{figure}

\subsection{Schnakenberg model}
\label{sec:ex4}
Last we consider the Schnakenberg model
which describes biological pattern
formation due to diffusion-driven instability \cite{hao2020spatial}:
\begin{equation}\left\{
	\begin{aligned}
		&\frac{\partial u}{\partial t} = \Delta u + \eta(a-u+u^2v),\\
		&\frac{\partial v}{\partial t} = d\Delta v + \eta(b-u^2v),
	\end{aligned}\right.\label{ex4}
\end{equation}
where $u$ is an activator and $v$ is a substrate. Here $d$ represents the relative dispersal rate of two species while the parameter $\eta$ specifies the relative balance between the dispersion and the chemical reaction where $v$ is converted to $u$ in a nonlinear way and $u$ decays linearly. The steady-state system of (\ref{ex4}) with the non-flux boundary condition has been well-studied in \cite{hao2020spatial} and shown multiple steady-state solutions and the bifurcation structure to the diffusion parameter $d$. 
	Although the supercritical pitchfork bifurcation is observed for parameter $d$, it is unclear for the bifurcation types for the whole parameter space by 
	treating  $a$, $b$, and $\eta$ as parameters. Thus we use EDNN to explore the bifurcations on the three dimensional parameter space. More specifically, we consider the discretized steady-state system on a 1D domain $x\in[0,1]$ with no-flux boundary conditions with the stepsize $h=0.1$.
We fix $a = 1/3$ and compute the bifurcation points of $d$ in terms of $b$ and $\eta$ via EDNN.

To train the one-hidden-layer EDNNs with the sigmoid activation function, we compute solutions for the system \eqref{ex4} with $a = 1/3$, $b_i =\frac{2}{3} +  \frac{i}{30}$ and $\eta_j = 45 + \frac{j}{2}$ for $i = 0,\cdots,4$ and $j = 0,\cdots,70$ with different values of $d$.  For each pair of $(b_i,\eta_j)$, we generate 300 training data points. To get the bifurcation approximation of $d$ for given $(a,b_i,\eta_j)$, we train 10 different EDNNs independently and compute the bifurcation point of $d$ for the best EDNN (with the smallest loss). The color map of bifurcation points of $d$ in term of $(b,\eta)$ is shown in Fig. \ref{Fig:Schnakenberg}. It shows that if the chemical reaction ($\eta$) is relatively small, the number of solutions get larger when the diffusion ($d$) is large.  But if the chemical reaction 
		is already large, the solution structure is complex even the diffusion is even small.

\begin{figure}
	\centering
	\includegraphics[width=0.45\linewidth]{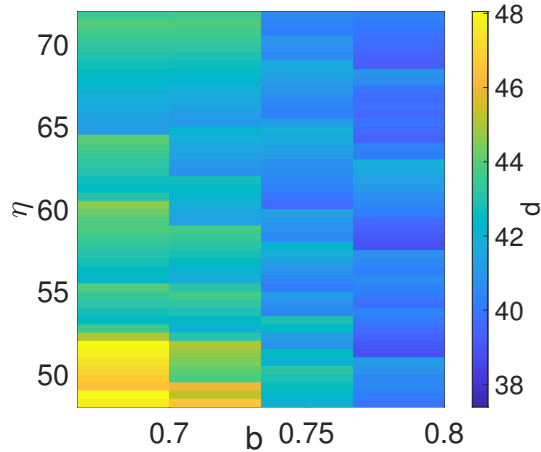}
	\caption{The bifurcation points of $d$ for given $(b,\eta)$ with $a=1/3$ for Schnakenberg model.}
	\label{Fig:Schnakenberg}
\end{figure}

\nocite{*}

\section{Conclusion}
In this paper, we develop a novel numerical method for computing bifurcation points of nonlinear parametric systems based on EDNNs.  This new approach is based on the correspondence between solutions of nonlinear systems and the parameter space and reveals this nonlinear correspondence by EDNNs which combines both solution data and the equation information. Then we compute the bifurcation points by using trained EDNNs to search the minimum eigenvalues on the parameter space. We provide both theoretical analysis and numerical examples for one-hidden-layer EDNNs.
Numerical results show that deep neural networks  are needed to accurately compute the bifurcation points. One of the future directions is to  utilize local information to refine the predictions of the bifurcation points via endgames, such as the power series endgame \cite{morgan1992power}. The idea is  to approximate the bifurcation point by using fractional power series expansion of the solution with respect to the bifurcation parameter.
		This has been successfully adapted to computing the
		location of critical points \cite{hauenstein2018semidefinite}.

Another future direction is to extend the current work to deep neural networks for both convergence analysis and training algorithms. More specifically, greedy training algorithm would be helpful to observe the convergence numerically \cite{hao2021efficient}.

Moreover, this proposed neural network approach can be extended to solve pattern formation problems in biology, physics, and engineering. These mathematical models of differential equations provide a rich source of computing multiple solutions of nonlinear models, for instance, multiple equilibria separated by saddle-node bifurcations in patchy ecosystems and electric power grids \cite{brummitt2015coupled,eppstein2012random}. 
		However, current numerical methods for computing multiple solutions have some limitations, for instance, it is hard to choose a good initial guess and the computations for 3D sometimes even 2D become very inefficient. The EDNN can provide an alternative tool by computing the bifurcations so that the multiple equilibria would be tracked separately which can reduce the computational cost and also provide a good initial guess. Another challenge in this field is how to compute accurate solutions of chaotic systems with incomplete information, e.g., noisy or incomplete solution information \cite{gelbrecht2021neural}. One scenario is the solution information could be incomplete near bifurcation points since the accuracy of Newton's solver is very low near bifurcation points. In the future, we will extend the training data with some perturbed noise to further test the robustness of EDNN.

\bibliographystyle{unsrt}  
\bibliography{dnn-bifur}  

\end{document}